\documentclass{amsart}
\usepackage{amsfonts,amssymb,amsmath,amsthm}
\usepackage{url}
\usepackage{enumerate}

\newcommand{\dueto}[1]{\textup{\textbf{(#1) }}}
\newcommand{\nin}{\not\in}
\newcommand{\nocomma}{}

\newcommand{\tmtextit}[1]{{\itshape{#1}}}

\newtheorem{theorem}{Theorem}[section]
\newtheorem{lemma}[theorem]{Lemma}
\newtheorem{proposition}[theorem]{Proposition}

\theoremstyle{definition}
\newtheorem{definition}[theorem]{Definition}

\numberwithin{equation}{section}

\title{A new construction of the degree of maximal monotone maps}
\author{Mohammad Niksirat}
\address{Department of Mathematics, University of Alberta, Edmonton, Canada, T6G 2J5}
\email{niksirat@ualberta.ca}
\keywords{degree theory, finite rank approximation, maximal monotone maps, multivalued maps}

\subjclass[2010]{Primary 47H11, Secondary 47H07}
\begin{document}

\begin{abstract}
The inclusion equations of the type $f \in T ( x)$ where $T : X \to 2^{X^{\ast}}$ is a maximal monotone map, are extensively studied in nonlinear analysis. In this paper, we present a new construction of the degree of maximal monotone maps of the form $T:Y\to 2^{X^*}$, where $Y$ is a locally uniformly convex and separable Banach space continuously embedded in $X$.  The advantage of the new construction lies in the remarkable simplicity it offers for calculation of degree in comparison with the classical one suggested by F. Browder. We prove a few classical theorems in convex analysis through the suggested degree.
\end{abstract}
\maketitle

\section{Introduction}

Assume that $X$ is a uniformly convex Banach space, $Y$ is separable and reflexive Banach spaces equipped with
uniformly convex norms, and $i : Y \rightarrow X$ is the continuous embedding. Furthermore, assume that $T : Y \rightarrow 2^{X^{\ast}}$ is a maximal monotone
map with the effective domain $D ( T) = Y$ in the following sense. A pair $(
\tilde{y}, \tilde{x}^{\ast})$ is in the graph of $T$ if the condition $\langle x^{\ast} -
\tilde{x}^{\ast}, i ( y - \tilde{y}) \rangle \geq 0$ holds for all $( y, x^{\ast})
\in {\rm graph} ( T)$ where $\langle, \rangle$ denotes the continuous pairing between $X,X^*$. In this article, we propose a topological degree $T$ possessing classical properties of topological degree in certain sense. The construction generalizes the F. Browder's classical degree of maximal monotone maps
\cite{Browder82}. The Browder's degree is constructed as follows.
Assume that $X$ is a reflexive Banach space equipped with a
uniformly convex norm and $T : X \rightarrow 2^{X^{\ast}}$ a maximal monotone
map. The map $T_{\epsilon} = T + \epsilon J$, for $\epsilon > 0$
where $J : X \rightarrow X^{\ast}$ is the duality mapping possesses the following properties:
\begin{enumerate}
  \item $T_{\epsilon}$ is a map satisfying the following condition: for any $x_n
  \rightharpoonup x$ in $X$, if there is $x^{\ast}_n \in T ( x_n)$ such that
  \[ \limsup_{n \rightarrow \infty} \langle x^{\ast}_n + \epsilon J ( x_n),
     x_n - x \rangle \leq 0, \]
  then $x_n \rightarrow x$.
  
  \item The map $T_{\epsilon}$ is onto $X^{\ast}$,
  
  \item if $x_1 \neq x_2$, the sets $T(x_1),T(x_2)$ are disjoint, that is,
  \[ T ( x_1) \cap T ( x_2) = \emptyset . \]
  \item The map $T_{\epsilon}^{- 1} : X^{\ast} \times ( 0, \infty) \rightarrow X$
  is well defined and continuous.
\end{enumerate}
Notice that $J$ is single valued, bijective and bi-continuous if $X$ is uniformly convex. It is shown  that the map $( T_{\epsilon}^{- 1} + \epsilon J^{- 1})^{- 1} : X
\rightarrow X^{\ast}$ is a single valued demi-continuous and $( S)_+$ for which a degree theory has been developed by F. Browder. The
degree of $T$ at $0\in X^*$ in an open bounded set $D \subset X$ is defined by the following relation

\begin{equation}
  \label{BD} \deg ( T, D, 0) = \lim_{\epsilon \rightarrow 0} \deg ( (
  T_{\epsilon}^{- 1} + \epsilon J^{- 1})^{- 1}, D, 0).
\end{equation}
The degree suggested in this article generalizes the Browder's degree in the sense that if $Y = X$, the two degrees are the same. The main advantage of the suggested degree is the direct use of finite rank approximation we employed in our previous work {\cite{Nik-deg1}} for single valued mappings. It is seen that the constructed degree makes calculations much simpler than (\ref{BD}). We take note that that the suggested degree is different from the degree of the map
$i^{\ast} \circ T : Y \rightarrow 2^{Y^{\ast}}$. The difference between two
formulations is discussed in {\cite{Nik-deg1}} for single valued maps.

\begin{definition}
  Assume that $X_1$ and $X_2$ are Banach spaces. A map $A : X_1 \rightarrow
  2^{X_2}$ is called upper semi-continuous at $x \in X_1$ if for every
  neighborhood $V$ of $A ( x)$, there exists an open neighborhood $U$ of $x$
  such that $T ( U) \subset V$.
\end{definition}

We have the following theorem for the upper semi-continuous multi-valued
mappings; see for example \cite{Cellina1969,Lloyd1978}.

\begin{theorem}
  \label{econt}{\dueto{$\epsilon$-continuous subgraph}}Assume that $X_1$
  and $X_2$ are Banach spaces, and the map $A : X_1 \rightarrow 2^{X_2}$ is
  upper semi-continuous. If $A ( x)$ is closed and convex for all $x \in X_1$,
  then for any $\epsilon > 0$, there exists a continuous single valued
  function $A_{\epsilon} : X_1 \rightarrow X_2$ such that for any $x \in
  X_1$, there exists $z_1 \in X_1$ and $\tilde{z}_2 \in A ( z)$ such that $\| x -
  z_1 \| < \epsilon$ and $\| A_{\epsilon} ( x) - \tilde{z}_2 \| <
  \epsilon$.
\end{theorem}

\begin{proposition}
  Let $X, Y$ be Banach spaces, $i : Y \rightarrow X$ a continuous embedding
  and $T : Y \rightarrow 2^{X^{\ast}}$ a maximal monotone map with the
  effective domain $Y$. Then $T ( y)$ is closed and convex for all $y \in
  Y$, and $T$ is norm to weak-star upper semi-continuous in the following
  sense. For arbitrary $y \in Y$, and arbitrary sequence $( y_n)$ converges to
  $y$ in norm, there is a weakly limit point $x^{\ast}$ of $\cup_n T ( y_n)$
  such that $x^{\ast} \in T ( y)$.
\end{proposition}

The proof is completely similar to one for the map $T : X \rightarrow
2^{X^{\ast}}$. For a proof of the standard version see for example \cite{Barbu}.

If $Y$ is a separable and reflexive Banach space equipped with a uniformly
convex norm, a theorem by Browder and Ton \cite{Browder1968a} guarantees the
existence of a separable Hilbert space $H$ such that the embedding $j : H
\hookrightarrow Y$ is dense and compact. Choosing an orthogonal basis $\{ h_k
\}_{k = 1}^{\infty}$ for $H$, we obtaine the basis $\mathcal{Y}= \{ y^1,
y^2, \cdots, y^n, \cdots \}$ for $Y$ where \ $y^k = j ( h_k)$, and accordingly, the filtration $\mathbb{Y}= \{ Y_n \}$, where $Y_n = {\rm span}
\{ y^1, \ldots, y^n \}$. The following proposition is simply verified.
\begin{proposition}
For any $y\in Y$, there is a sequence $(y_n), y_n\in Y_n$ such that $y_n\to y$.
\end{proposition}
The pairing in $Y_n$ is denoted by $(,)$ and is defined by the relation $(y^i,y^j)=\delta_{i j}$ for all $y^i, y^j\in \mathcal{Y}$. We define the maximal monotone operato $T:Y\to 2^{X^*}$ in the following sense.
\begin{definition}
  Suppose $X$ and $Y$ are separable and reflexive Banach spaces equipped with
  uniformly convex norm, and assume that $T : Y \rightarrow 2^{X^{\ast}}$ is a maximal monotone map. The finite rank approximation of arbitrary $x^{\ast}
  \in T ( y)$ in $Y_n \in \mathbb{Y}$, is defined by $\hat{x}_n = \sum_{k =
  1}^n \langle x^{\ast}, i ( y^k) \rangle y^k$. Accordingly, the finite rank
  map $T_n : Y \rightarrow 2^{Y_n}$ is defined by the relation
  \begin{equation}
    \label{Gen-mmon-ap} T_n ( y) = \bigcup_{x^{\ast} \in T ( y)} \hat{x}_n .
  \end{equation}
\end{definition}
For any $x^*\in X^*	$ and $y \in Y_n$, we have the property
\[ ( \hat{x}_n, y) = \langle x^{\ast}, i ( y) \rangle, \]
where $\langle, \rangle$ is the
pairing between $X^{\ast}, X$. In fact, if $x^{\ast} \in X^*$, then for
$\hat{x}_n = \sum_{k = 1}^n \langle x^{\ast}, y^k \rangle y^k$, we have
\[ \sum_{k = 1}^n \langle x^{\ast}, i ( y^k) \rangle ( y^k, y) = \sum_{k =
   1}^n \langle x^{\ast}, i ( ( y^k, y) y^k) \rangle = \left\langle x^{\ast},
   i \left( \sum_{k = 1}^n ( y^k, y) y^k \right) \right\rangle, \]
and thus $\langle x^{\ast}, i ( y) \rangle = ( \hat{x}_n, y)$.

\begin{lemma}
  \label{Tnlem}The finite rank approximation $T_n$ is upper semi-continuous
  and for every $x \in Y$, the set $T_n (x)$ is closed and convex.
\end{lemma}

\begin{proof}
  Fix $n$ and $\epsilon > 0$. If $T_n$ is not upper semi-continuous at $x
  \in Y$, there is a sequence $( \delta_m), \delta_m \rightarrow 0$ and $x_m
  \in B_{\delta_m} ( x)$ such that for some $\hat{x}_{n, m} \in T_n ( x_m)$,
  we have $\hat{x}_{n, m} \nin V_{\epsilon} ( T_n ( x))$. $T$ is maximal
  monotone, and thus locally bounded. Therefore, there is a subsequence (shown
  for the sake of simplicity again by $\hat{x}_{n, m}$) such that $\hat{x}_{n, m}
  \to \hat{x}$. We show $\hat{x} \in T_n ( x)$. Since $\hat{x}_{n,
  m} \in T_n ( x_m)$, there is $x^{\ast}_m \in T ( x_m)$ such that
  $\hat{x}_{n, m}$ are the finite rank approximation in $Y_n$ of $x^{\ast}_m$,
  that is,
  \[ \hat{x}_{n, m} = \sum_{k = 1}^n \langle x^{\ast}_m, i ( y^k) \rangle y^k
     . \]
  Since $T$ is norm to weak-star upper semi-continuous, and $x_m \rightarrow
  x$, we have $x^{\ast}_m \rightharpoonup x^{\ast}$ for some $x^{\ast} \in T (
  x)$. Thus $\langle x^{\ast}_m, i ( y^k) \rangle \rightarrow \langle
  x^{\ast}, i ( y^k) \rangle$ for all $1 \leq k \leq n$. Therefore
  \[ \hat{x}_{n, m} = \sum_{k = 1}^n \langle x^{\ast}_m, i ( y^k) \rangle y^k
     \rightarrow \sum_{k = 1}^n \langle x^{\ast}, i ( y^k) \rangle y^k \in T_n
     ( x), \]
  and therefore $\hat{x} \in T_n ( x)$.
  Now we show  that $T_n ( x)$ is closed for all $x\in Y$. Consider an arbitrary sequence $\hat{x}_m \in T_n
  (x)$, and $\hat{x}_m \rightarrow \hat{x}$. Let $x^{\ast}_m \in T ( y)$ be
  the sequence such that $\hat{x}_m = \sum_{k = 1}^n \langle x^{\ast}_m, i (
  y^k) \rangle y^k$. Since $T ( y)$ is bounded and convex, the sequence
  $x^{\ast}_m$ converges weakly (in a subsequence) to some $x^{\ast} \in T (
  x)$ and thus
  \[ \hat{x}_m = \sum_{k = 1}^n \langle x^{\ast}_m, i ( y^k) \rangle y^k
     \rightarrow \sum_{k = 1}^n \langle x^{\ast}, i ( y^k) \rangle y^k, \]
  and thus
  \[ \hat{x} = \sum_{k = 1}^n \langle x^{\ast}, i ( y^k) \rangle y^k \in T_n
     ( x) . \]
  That $T_n ( x)$ is convex follows simply from the convexity of $T ( x)$.
\end{proof}

By the Lemma (\ref{Tnlem}) and Theorem (\ref{econt}), the $\epsilon$-continuous selection $T_{n, \epsilon}$ of $T_n$ is well defined. The single valued map $T_{n, \epsilon}$ is continuous and for any
$x \in Y_n$, there is $z \in Y_n$ and $\hat{z} \in T_n ( z)$ such that $\|  z-x \| \in \epsilon$ and $\| \hat{z} - T_{n, \epsilon} (x) \|
< \epsilon$.

\section{Degree definition}

Let $( \epsilon_n)$ be a positive sequence such that $\epsilon_n \rightarrow
0$. Fix $\epsilon > 0$. Consider the function $\tilde{T}_{n, \epsilon_n}
: Y_n \rightarrow Y_n$ defined by the relation
\begin{equation}
  \label{tildeT} \tilde{T}_{n, \epsilon_n} = T_{n, \epsilon_n} +
  \epsilon J_n,
\end{equation}
where $T_{n, \epsilon_n}$ is the $\epsilon_n$-continuous selection of
$T_n$ and $J_n$ is the finite rank approximation of $J \circ i : Y \rightarrow
X^{\ast}$ in $Y_n$ where $J:X\to X^*$ is the bi-continuous duality map.

\begin{lemma}
  Let $D \subset Y$ be an open bounded set and assume that for some
  $\epsilon > 0$, we have $0 \nin {\rm cl}  T^{\epsilon} ( \partial
  D)$, where $T^{\epsilon} = T + \epsilon J \circ i$. Then there is $N > 0$ such that $0
  \nin \tilde{T}_{n, \epsilon_n} ( \partial D_n)$ for all $n \geq N$ where
  $D_n = D \cap Y_n$.
\end{lemma}

\begin{proof}
  Otherwise, there is a sequence $z_n \in \partial D_n$ such that
  $\tilde{T}_{n, \epsilon_n} ( z_n) = 0$ for all $n \geq 1$. Since
  $\partial D$ is bounded, there is a subsequence (we show again by $z_n$)
  that weakly converges to $z$. We first show that $z_n$ converges strongly to
  $z$. Choose a sequence $\zeta_n \in Y_n$ that converges to $z$ in norm.
  Since $\tilde{T}_{n, \epsilon_n} ( z_n) = 0$ on $Y_n$, we have
  \[ ( T_{n, \epsilon_n} ( z_n), z_n - \zeta_n) + \epsilon ( J_n ( z_n),
     z_n - \zeta_n) = 0, \]
  because $z_n - \zeta_n \in Y_n$. By the relation
  \[ ( J_n ( z_n), z_n - \zeta_n) = \langle J \circ i ( z_n), i ( z_n -
     \zeta_n) \rangle, \]
  we can write
  \begin{equation}
    \label{Tne} ( T_{n, \epsilon_n} ( z_n), z_n - \zeta_n) + \epsilon
    \langle J \circ i ( z_n), i ( z_n - \zeta_n) \rangle = 0.
  \end{equation}
  On the other hand, for each $z_n$, there is $x_n \in Y_n$ and $\hat{x}_n \in
  T_n ( x_n)$ such that
  \[ \| x_n-z_n \| < \epsilon_n, \| \hat{x}_n - T_{n, \epsilon_n} (
     z_n) \| < \epsilon_n . \]
  Therefore, we have
  \[ ( T_{n, \epsilon_n} ( z_n), z_n - \zeta_n) = ( T_{n, \epsilon_n} (
     z_n) - \hat{x}_n, z_n - \zeta_n) + ( \hat{x}_n, z_n - \zeta_n),
      \]
      and by the relation $\|T_{n, \epsilon_n} (
     z_n) - \hat{x}_n\| < \epsilon_n$, we obtain
      \[( T_{n, \epsilon_n} ( z_n), z_n - \zeta_n)
       \geq -
     \epsilon_n \| z_n - \zeta_n \| + ( \hat{x}_n, z_n - \zeta_n)
      \]
  By the relation $\| x_n - z_n \| < \epsilon_n$, we have
  \[ ( \hat{x}_n, z_n - \zeta_n) \geq - \epsilon_n \| \hat{x}_n \| + (
     \hat{x}_n, x_n - \zeta_n) . \]
  Since $\hat{x}_n \in T_n ( x)$, there are $x^{\ast}_n \in T ( x_n)$ such
  that $\hat{x}_n$ are the finite rank approximations of $x^{\ast}_n$ in
  $Y_n$. Thus, we can write
  \[ ( \hat{x}_n, x_n - \zeta_n) = \langle x^{\ast}_n, i ( x_n - \zeta_n)
     \rangle . \]
  Also, for some $C > 0$, we can write
  \[ \langle x^{\ast}_n, i ( x_n - \zeta_n) \rangle \geq - C \| x^{\ast}_n \| 
     \| z - \zeta_n \| + \langle x^{\ast}_n, i ( x_n - z) \rangle . \]
  Choose an arbitrary $z^{\ast} \in T ( z)$. We have
  \[ \langle x^{\ast}_n, i ( x_n - z) \rangle \geq \langle x^{\ast}_n -
     z^{\ast}, i ( x_n - z) \rangle + \langle z^{\ast}, i ( x_n - z) \rangle
     \geq \langle z^{\ast}, i ( x_n - z) \rangle . \]
  Since $z_n \rightharpoonup z$ and $\| x_n - z_n \| \to 0$, we conclude
  \[ \lim_{n \rightarrow \infty} ( T_{n, \epsilon_n} ( z_n), z_n - \zeta_n)
     \geq 0. \]
  Thus, the relation (\ref{Tne}) implies
  \[ \limsup_{n \rightarrow \infty}  \langle J \circ i ( z_n), i ( z_n -
     \zeta_n) \rangle \leq 0. \]
  By the relation $\zeta_n \rightarrow z$, we conclude
  \[ \limsup_{n \rightarrow \infty}  \langle J \circ i ( z_n), i ( z_n -
     \zeta_n) \rangle \leq 0, \]
  and since $J$ is a map of class $( S)_+$, we obtain $z_n \rightarrow z \in
  \partial D$. 
  Now we show $0 \in {\rm cl}  T^{\epsilon} ( z)$. Choose
  arbitrary $y \in Y$ and sequence $y_n \in Y_n$, $y_n \rightarrow y$. We have
  \[ ( T_{n, \epsilon_n} ( z_n), y_n) + \epsilon \langle J \circ i (
     z_n), i ( y_n) \rangle = 0. \]
  Since $J$ is continuous, we have $\langle J \circ i ( z_n), i ( y_n) \rangle
  \rightarrow \langle J \circ i ( z), y \rangle$. Choose $x_n \in Y_n$ and
  $\hat{x}_n \in T_n ( x_n)$ such that
  \[ \| x_n-z_n \| < \epsilon_n, \| \hat{x}_n - T_{n, \epsilon_n} (
     z_n) \| < \epsilon_n . \]
  We have
  \[ \lim | ( T_{n, \epsilon_n} ( z_n), y_n) - ( \hat{x}_n, y_n) | = 0. \]
  For $x^{\ast}_n \in T ( x_n)$, and by the relation $y_n \rightarrow y$, we
  obtain
  \[ \lim | ( T_{n, \epsilon_n} ( z_n), y_n) - \langle x^{\ast}_n, i ( y)
     \rangle | = 0. \]
  Since $T$ is norm to weak-star upper semi-continuous, we conclude $\langle
  x^{\ast}_n, i ( y) \rangle \rightarrow \langle x^{\ast}, y \rangle$ for some
  $x^{\ast} \in T ( z)$. This implies that $x^{\ast} + \epsilon J \circ i (
  z)$=0 and thus $0 \in {\rm cl}  T^{\epsilon} ( z)$ that contradicts the
  condition $0 \nin {\rm cl}  T^{\epsilon} ( \partial D)$.
\end{proof}

\begin{proposition}
  Assume that $0 \nin {\rm cl}  T ( \partial D)$. Then there is
  $\epsilon > 0$ such that $0 \nin {\rm cl}  T^{\epsilon} ( \partial
  D)$.
\end{proposition}

\begin{proof}
  By the assumption, there is $r > 0$ such that ${\rm dist} ( 0 \nocomma,
  {\rm cl} T ( \partial D)) = r$. Let $z \in \partial D$ is arbitrary. Take
  arbitrary $z^{\ast} \in T ( z)$. We have
  \[ \| z^{\ast} + \epsilon J ( z) \| \geq \| z^{\ast} \| - \epsilon \|
     z \| \geq r - \epsilon \| z \| . \]
  Therefore $0 \nin {\rm cl}  T^{\epsilon} ( \partial D)$ if
  \begin{equation}
    \label{epsilon} 0 < \epsilon < \frac{r}{_{} \max_{z \in \partial D}  \|
    z \|} .
  \end{equation}
  The boundedness of $\partial D$ guarantees the existence of $\epsilon > 0$.
\end{proof}

\begin{definition}
  Assume that $X$ and $Y$ are separable and reflexive Banach spaces equipped
  with uniformly convex norms, $D \subset Y$ is an open bounded set and $T :
  Y \to 2^{X^{\ast}}$ is a maximal monotone map such that $0 \nin
  {\rm cl} T ( \partial D)$. Choose $\epsilon > 0$ satisfying $\left(
  \ref{epsilon} \right)$ and consider the map $\tilde{T}_{n, \epsilon_n}$
  defined in $\left( \ref{tildeT} \right)$. The degree of $T$ in $D$ with
  respect to $0$ is defined by the following formula
  \begin{equation}
    \label{mmdeg} \deg ( T, D, 0) = \lim_{n \rightarrow \infty} \deg_B (
    \tilde{T}_{n, \epsilon_n}, D_n, 0),
  \end{equation}
  where $\deg_B$ is the usual Brouwer's degree of the map $\tilde{T}_{n,
  \epsilon_n}$.
\end{definition}

\begin{proposition}
  The degree defined in the relation $\left( \ref{mmdeg} \right)$ is
  stable with respect to $n$.
\end{proposition}

\begin{proof}
  Consider the sequence of mappings $( \tilde{T}_{k, \epsilon_k})$ such that for
  sufficiently large $n$ the condition $0
  \nin {\rm cl} \tilde{T}_{k, \epsilon_k} ( \partial D_k)$ is satisfied for $k \geq n - 1$ .
  First note that \ there is $\epsilon_0 > 0$ such that for $0 <
  \epsilon_1, \epsilon_2 < \epsilon_0$, the following relation holds
  \begin{equation}
    \label{e1e2deg} \deg_B ( \tilde{T}_{n, \epsilon_1}, D_n, 0) = \deg_B (
    \tilde{T}_{n, \epsilon_2}, D_n, 0) \nocomma,
  \end{equation}
  In fact, for any $x \in Y_n$, there is $z_1, z_2 \in Y_n$ and
  $\hat{z}_1 \in T_n ( z_1)$,  $\hat{z}_2 \in T_n ( z_2)$ such that
  \[ \| z_1 - x \| < \epsilon_1, \| \hat{z}_1 - T_{n, \epsilon_1} ( x)
     \| < \epsilon_1, \| z_2 - x \| < \epsilon_2, \| \hat{z}_2 - T_{n,
     \epsilon_2} ( x) \| < \epsilon_2 . \]
  The continuity of $T_{n, \epsilon_1}$, $T_{n, \epsilon_2}$ implies
  that $\| \tilde{T}_{n, \epsilon_1} - \tilde{T}_{n, \epsilon_2} \|$
  can be controlled and thus (\ref{e1e2deg}) holds. Let us write
  $\tilde{T}_{n, \epsilon_n}$ as
  \[ \tilde{T}_{n, \epsilon_n} = \tilde{T}_{n, \epsilon_n}^1 +
     \tilde{T}_{n, \epsilon_n}^2, \]
  where $\tilde{T}_{n, \epsilon_n}^1$ is the projection of $\tilde{T}_{n,
  \epsilon_n}$ into $Y_{n - 1}$ and $\tilde{T}_{n, \epsilon_n}^2$ is the
  projection into $\{ y^n \}$. Define the map $S_{n, \epsilon_n} : Y_n
  \rightarrow 2^{Y_n}$ as
  \[ S_{n, \epsilon_n} ( x) = \tilde{T}_{n, \epsilon_n}^1 ( x) + ( x,
     y^n) y^n \]
  Obviously, we have
  \begin{equation}
    \label{S1} \deg_B ( S_{n, \epsilon_n}, D_n, 0) = \deg_B ( \tilde{T}_{n,
    \epsilon_n}^1, D_{n - 1}, 0) .
  \end{equation}
  First we show
  \begin{equation}
    \label{3eq} \deg_B ( \tilde{T}_{n, \epsilon_n}^1, D_{n - 1}, 0) =
    \deg_B ( \tilde{T}_{n - 1, \epsilon_n}, D_{n - 1}, 0) = \deg_B (
    \tilde{T}_{n - 1, \epsilon_{n - 1}}, D_{n - 1}, 0) .
  \end{equation}
  The last equality follows from (\ref{e1e2deg}). In order to prove the first equality, we note
  that if $T_{n, \epsilon_n}$ is an $\epsilon_n$-continuous selection of
  $T_n$, then $T_{n, \epsilon_1}^1$ is also an $\epsilon_n$-continuous
  selection of $T_{{n - 1}, \epsilon_n}$. In fact, let $x \in Y_{n - 1}$ be arbitrary, then
  there is $z \in Y_n$ and $\hat{z} \in T_n ( y)$ such that
  \[ \| \hat{z}^1 - T^1_{n, \epsilon_n} ( x) + \hat{z}^2 - T^2 n,
     \epsilon_n \|^2 < \epsilon_n^2, \]
  where $\hat{z}^1 \in Y_{n - 1}$ and $\hat{z}^2 \in \{ y^n \}$. This implies
  \[ \| \hat{z}^1 - T^1_{n, \epsilon_n} ( x) \| < \epsilon_n, \| z^1 -
     y \| < \epsilon_n . \]
  Again it follows that $\| T_{n - 1, \epsilon_n} - T^1_{n, \epsilon_n} \|$ can
  be controlled and thus the first equality in (\ref{3eq}) is proved. Now,
  we show
  \begin{equation}
    \label{S2} \deg_B ( \tilde{T}_{n, \epsilon_n}, D_n, 0) = \deg_B ( S_{n,
    \epsilon_n}, D_n, 0) .
  \end{equation}
  Consider the convex homotopy
  \[ h_n ( t) = ( 1 - t)  \tilde{T}_{n, \epsilon_n} + t S_{n,
     \epsilon_n} . \]
  It is enough to show $0 \nin h_n ( t) ( \partial D_n)$ for $t \in [ 0, 1]$.
  Clearly, $0 \nin h_n ( t) ( \partial D_n)$ for $t = 0, 1$. For $t \in ( 0,
  1)$ assume that there exists a sequence $t_n \in ( 0, 1)$ and $( z_n)$, $z_n
  \in \partial D_n$ such that $h_n ( t_n) ( z_n) = 0$. According to the
  construction of $h_n ( t)$ we have
  \[ 0 = h_n ( t_n) ( z_n) = \tilde{T}^1_{n, \epsilon_n} ( z_n) + ( 1 -
     t_n)  \tilde{T}^2_{n, \epsilon_n} ( z_n) + t_n  ( z_n, y^n) y^n . \]
  The above relation implies $\tilde{T}_{n, \epsilon_n}^1 ( z_n) = 0$ and
  \[ \tilde{T}^2_{n, \epsilon_n} ( z_n) = - \frac{t_n}{1 - t_n}  ( z_n,
     y^n) y^n . \]
  Since $\partial D$ is bounded then $z_n \rightharpoonup z$ in a
  subsequence. Choose the sequence $( \zeta_n), \zeta_n \in Y_n$ and $\zeta_n
  \rightarrow z$ and obtain
  \[ ( \tilde{T}_{n, \epsilon_n} ( z_n), z_n - \zeta_n) = - \frac{t_n}{1 -
     t_n}  | ( z_n, y^n) |^2 + \frac{t_n}{1 - t_n}  ( z_n, y^n) ( \zeta_n,
     y^n) . \]
  On the other hand since $\zeta_n \rightarrow z$, we have $( \zeta_n, y^n)
  \rightarrow 0$. Since there exists sequence $\hat{z}_n \in T_n ( z_n)$ such
  that $\| \hat{z}_n - T_{n, \epsilon_n} ( z_n) \| < \epsilon_n$ we can write for some $z^{\ast}_n \in T ( z_n)$
  \[ \limsup_{n \rightarrow \infty}  \langle z^{\ast}_n + \epsilon J \circ
     i ( z_n), i ( z_n - z) \rangle = \limsup_{n \rightarrow \infty}  (
     \tilde{T}_{n, \epsilon_n} ( z_n), z_n - \zeta_n) . \]
  Therefore we obtain
  \[ \limsup_{n \rightarrow \infty}  \langle z^{\ast}_n + \epsilon J \circ
     i ( z_n), i ( z_n - z) \rangle \leq 0. \]
  Since
  \[ \lim_{n \rightarrow \infty} \langle z^{\ast}_n, i ( z_n - z) \rangle \geq
     0, \]
  we obtain
  \[ \limsup_{n \rightarrow \infty}  \langle J \circ i ( z_n), i ( z_n - z)
     \rangle \leq 0, \]
  and thus $z_n \rightarrow z$. This is impossible because $0 \nin {\rm cl} 
  T^{\epsilon} ( \partial D)$.
\end{proof}

Now, we show that the definition (\ref{mmdeg}) satisfies the classical
properties of a topological degree including the solvability and the homotopy invariance.

\begin{theorem}
  Let $D \subset Y$ be an open bounded set and assume that $T : Y \rightarrow
  2^{X^{\ast}}$ is maximal monotone and $0 \nin {\rm cl} T ( \partial D)$. If
  \[ \deg ( T, D, 0) \neq 0, \]
  then there is $y \in D$ such that $0 \in T ( y)$.
\end{theorem}

\begin{proof}
  Assume $\deg ( T, D, 0) \neq 0$, then there exists a sequence $z_n \in D$ such that
  $\tilde{T}_{n, \epsilon_n} ( z_n) = 0$ for sufficiently small $\epsilon_n>0$.
  This implies that there is the sequence $\hat{z}_n \in T_n ( z_n)$ such that $\|
  \hat{z}_n + \epsilon J_n ( z_n) \| < \epsilon_n$. Since $D$ is bounded
  then $z_n$ converges weakly (in a subsequence) to some $z$. Since $T$ is
  monotone, we conclude
  \[ \limsup_{n \rightarrow \infty} \langle J \circ i ( z_n), i ( z_n - z)
     \rangle \leq 0, \]
  and thus $z_n$ converges strongly to $z \in {\rm cl} (D)$. Let
  $\hat{z}_n$ be the $n$-approximation of $z^{\ast}_n \in T ( z_n)$. Since
  $T$ is norm to weak-star upper semi-continuous, the sequence $z^*_n$ converges
  weakly in a subsequence to some $z^{\ast} \in T ( z)$. Let $v \in Y$ be
  arbitrary. Consider the sequence $( v_n), v_n \in Y_n$ and $v_n \rightarrow
  y$. Then we have
  \[ \langle z^{\ast}_n, i ( v_n) \rangle = - \epsilon \langle J \circ i (
     z_n), i ( v_n) \rangle - \epsilon ( T_{n, \epsilon_n} ( z_n) -
     \hat{z}_n, v_n \rangle \rightarrow 0, \]
  On the other hand, $\langle z^{\ast}_n, i ( v_n) \rangle \rightarrow
  \langle z^{\ast}, i ( y) \rangle$, and thus $z^{\ast} = 0$ or equivalently
  $0 \in T ( z)$. Since $0 \nin {\rm cl} T ( \partial D)$, we conclude $z \in
  D$.
\end{proof}

\begin{definition}
  Let $D \subset X$ be an open bounded set, and assume $h : [ 0, 1] \times Y
  \rightarrow 2^{X^{\ast}}$ be a continuous homotopy with respect to $t$ such
  that for any $t \in [ 0, 1]$, the map $h ( t) : Y \rightarrow 2^{X^{\ast}}$
  is maximal monotone. Furthermore assume that $0 \nin {\rm cl} h ( [ 0, 1]\times
  \partial D)$. The map $h$ is called an admissible homotopy for maximal
  monotone maps.
\end{definition}

\begin{proposition}
  The degree defined in $\left( \ref{mmdeg} \right)$ is stable under the
  admissible homotopy of maximal monotone maps.
\end{proposition}

\begin{proof}
  According to the definition of the admissible homotopy, the degree \[\deg_B (
  \tilde{h}_{n, \epsilon_n} ( t), D_n, 0)\] is independent of $t$ due to the fact $0 \nin
  \tilde{h}_{n, \epsilon_n} ( t) ( \partial D_n)$ for $t \in [ 0, 1]$ and the
  homotopy invariance of the Brouwer's degree. Now,
  the stability of the defined degree $\left( \ref{mmdeg} \right)$ with
  respect to $n$ implies that the degree $\deg ( h ( t), D, 0)$ is independent
  of $t$.
\end{proof}

\section{Degree theoretic proofs}

We give degree theoretic proofs of some theorems in convex analysis. The first
theorem is due to D. DeFigueirdo \cite{Figuerido1971}.

\begin{theorem}
  Assume that $X$ is a separable uniformly convex Banach space, $T : X
  \rightarrow 2^{X^{\ast}}$ is a maximal monotone map such that $0 \nin (T +
  \lambda J) (S_r)$, where $S_r$ is the sphere of radius $r$ and $\lambda > 0$
  is arbitrary. Then there exists $u \in {\rm cl} ( B_r)$ such that $0 \in T
  ( u)$, where $B_r$ is the ball of radius $r$ in $X$.
\end{theorem}

\begin{proof}
  Assume that $0 \nin T ({\rm cl} ( B_r))$. We show first that $0 \nin
  {\rm cl} T (S_r)$. Otherwise there exist a sequence $u_n \in S_r$ and $u^{\ast}_n \in
  T ( u_n)$ such that $u^{\ast}_n \rightarrow 0$. The sequence $( u_n)$
  converges weakly in a subsequence (that we show again by $u_n$) to some $u
  \in {\rm cl} ( B_r) .$ Claim: $[u, 0] \in {\rm graph} (T)$. For any $[x,
  x^{\ast}] \in {\rm graph} (T)$ we have the inequality
  \[ \langle x^{\ast}, x - u \rangle = \lim \langle x^{\ast} - u^{\ast}_n, x -
     u_n \rangle \geq 0. \]
  Since $T$ is maximal monotone, the above inequality implies $[x, 0] \in
  {\rm graph} (T)$ or equivalently $0 \in T ( u)$. This contradicts the
  assumption $0 \nin T ({\rm cl} ( B_r))$. It is also apparent that $0 \nin {\rm cl} ( J
  (S_r))$. Next, we show $0 \nin {\rm cl} ( (1 - t) T + t J) (S_r))$ for $t \in (0, 1)$. Otherwise there exist $t_n \in ( 0, 1)$,
  $u_n \in S_r$ and $u^{\ast}_n \in T ( u_n)$ such that
  \[ (1 - t_n) u^{\ast}_n + t_n \nocomma J ( u_n) \rightarrow 0. \]
  Again for $u_n \rightharpoonup u$ and $t_n \rightarrow t$ we obtain by the
  monotonicity property of $T$ the following inequality
  \[ \limsup_{n \rightarrow \infty} \langle J ( u_n), u_n - u \rangle \leq 0,
  \]
  that implies $u_n \rightarrow u \in S_r$. Claim: $[u, \frac{- t}{1 - t} J (
  u)] \in {\rm graph} (T)$. For any $[x, x^{\ast}] \in {\rm graph} (T)$ we
  obtain by the fact \ $\frac{- t}{1 - t} J ( u_n) \in T ( u_n)$ the following
  relation
  \[ \langle x^{\ast} + \frac{t}{1 - t} J ( u), x - u \rangle = \lim \langle
     x^{\ast} + \frac{t}{1 - t} J ( u_n), x - u_n \rangle \geq 0, \]
  that proves the claim. Now by degree theoretic argument we have
  \[ \deg (T, B_r, 0) = \deg ((1 - t) T + tJ, B_r, 0) = \deg (J, B_r, 0) = 1.
  \]
  The above calculation guarantees the existence of $u \in B_r$ such that $0
  \in T ( u)$ and this contradicts the assumption $0 \nin T ({\rm cl} (
  B_r))$. Therefore the assumption $0 \nin T ({\rm cl} ( B_r))$ is
  wrong and thus $0 \in T ({\rm cl} ( B_r))$.
\end{proof}

The next theorem is again from DeFigueirdo {\cite{Figuerido1971}}.

\begin{proposition}
  Let $X$ be a separable uniformly convex Banach space and assume that $f : X \rightarrow
  X^{\ast}$ is a pseudo-monotone map. Then ${\rm Rang} (\partial N_r + f) =
  X^{\ast}$ where $N_r$ is the map
  \[ N_r (x) = \left\{ \begin{array}{l}
       0 \hspace{1cm} {\rm if\  } x \in B_r\\
       1 \hspace{1cm} {\rm if\ } x \in S_r
     \end{array} \right., \]
  and $\partial N_r$ is the set of the sub-gradients of $N_r$.
\end{proposition}

\begin{proof}
  Apparently, we have
  \begin{equation}
    \partial N_r (x) = \left\{ \begin{array}{ll}
      0 & {\rm if\ }x \in B_r\\
      \{\lambda J (x), \lambda \geq 0\} & {\rm if\ }x \in S_R
    \end{array} \right. .
  \end{equation}
  Claim: for every $f_0 \in X^{^{\ast}}$, we have
  \begin{equation}
    \deg (\partial N_r + f - f_0, B_r, 0) \neq 0.
  \end{equation}
  First we show if $0 \nin (\partial N_r + f - f_0) ({\rm cl} ( B_r))$ then
  \begin{equation}
    0 \nin {\rm cl} (\partial N_r + f - f_0) (S_r) .
  \end{equation}
  Otherwise, there is a sequence $u_n \in S_r$ and $u^{\ast}_n \in \partial N_r (
  u_n)$ such that $u^{\ast}_n + f ( u_n) - f_0 \rightarrow 0$. But $u_n
  \rightharpoonup u \in {\rm cl} ( B_r)$ in a sub-sequence. We prove that
  $[u, f_0 - f ( u)] \in {\rm graph} (\partial N_r)$. Let $f_0 = u^{\ast}_n +
  f ( u_n) + \epsilon ( n)$ where $\epsilon ( n) \in X^{\ast}$ and
  $\epsilon ( n) \rightarrow 0$. For any arbitrary $[x, x^{\ast}] \in
  {\rm graph} (\partial N_r)$, we have
  \begin{align*}
  \langle x^{\ast} - f_0 + f ( u), x - u \rangle = \lim \langle x^{\ast} +
     f ( u) - u^{\ast}_n - f ( u_n), x - u_n \rangle  \geq \\ \lim \langle f ( u)
     - f ( u_n), x - u \rangle . 
   \end{align*}
  But
  \begin{equation}
   0 = \lim \langle u^{\ast}_n + f ( u_n) - f_0, u_n - u \rangle \geq \limsup
    \langle f ( u_n), u_n - u \rangle  
  \end{equation}
  Since $f$ is pseudo-monotone we obtain $f ( u_n) \rightharpoonup f ( u)$ and
  therefore
  \[ \langle x^{\ast} - f_0 + f ( u), x - u \rangle \geq 0. \]
  This implies that $[u, f_0 - f ( u)] \in {\rm graph} (\partial N_r)$ and
  thus $0 \in (\partial N_r + f - f_0) ({\rm cl} ( B_r))$ which is impossible
  by the assumption. Now consider the affine homotopy
  \begin{equation}
    h ( t) = (1 - t) (\partial N_r + f - f_0) + t \nocomma J.
  \end{equation}
  for $t \in (0, 1]$. We show
  \[ 0 \nin {\rm cl} ((1 - t) (\partial N_r + f - f_0) + t \nocomma J) (S_r)
     . \]
  Otherwise, there is a sequence $u_n \in S_r$, $u^{\ast}_n \in \partial N_r ( u_n)$
  and $t_n \rightarrow t$ such that
  \[ (1 - t_n) (u^{\ast}_n + f ( u_n) - w) + t_n Ju_n \rightarrow 0. \]
  But $u_n \rightharpoonup u \in {\rm cl} ( B_r)$ in a subsequence. W show
  \[ [u, f_0 - f ( u) - \frac{t}{1 - t} J ( u)] \in {\rm graph} (\partial
     N_r) . \]
  For any $[x, x^{\ast}] \in {\rm graph} (\partial N_r)$ we have
  \begin{eqnarray*}
    \langle x^{\ast} - f_0 + f ( u) + \frac{t}{1 - t} J ( u), x - u \rangle =
    &  & \\
    \lim \langle x^{\ast} + f ( u) + \frac{t}{1 - t} J ( u) - u^{\ast}_n - f (
    u_n) - \frac{t_n}{1 - t_n} J ( u_n), x - u \rangle \geq &  & \\
    \geq \limsup \langle f ( u) - f ( u_n), x - u \rangle + \liminf \langle
    \frac{t}{1 - t} J ( u) - \frac{t_n}{1 - t_n} J ( u_n), x - u \rangle &  & 
  \end{eqnarray*}
  We conclude $u_n \rightarrow u \in S_r$ and $f ( u_n) \rightharpoonup f (
  u)$ because $f$ is pseudo-monotone. Therefore we obtain again
  \[ 0 \in (\partial N_r + f + \frac{t}{1 - t} J - f_0) ({\rm cl} ( B_r)) .
  \]
  But $(\partial N_r + f + \frac{t}{1 - t} J - f_0) ({\rm cl} ( B_r)) =
  (\partial N_r + f - f_0) ({\rm cl} ( B_r))$ and then $0 \in (\partial N_R +
  f - f_0) ({\rm cl} ( B_r))$ that is impossible. Finally we use the homotopy
  invariance property of degree and write
  \[ \deg (\partial N_R + f - f_0, B_r, 0) = \deg (h ( t), B_R, 0) = \deg (J,
     B_r, 0) = 1. \]
  Therefore there exist $u \in {\rm cl} ( B_r)$ such that $f_0 \in \partial
  N_r ( u) + f ( u)$.
\end{proof}

The following theorem is due to F. Browder {\cite{Browder1976}} for the
surjectivity of the monotone maps with locally bounded inverse.

\begin{theorem}
  Assume $A : X \rightarrow X^{\ast}$ is a demi-continuous monotone map such
  that $A^{- 1}$ is locally bounded, that is, for every $f \in X^{\ast}$ there
  is a bounded $V_f \ni f$ such that $A^{- 1} (V_f)$ is bounded. Then $A$
  is onto.
\end{theorem}

\begin{proof}
  For any $f \in X^{\ast}$, we show that there is sufficiently large $r = r
  ( f)$ such that:
  \[ \deg (A, B_r, f) \neq 0. \]
  Choose $r > 0$ such that for a neighborhood $V_f \ni f$ the following condition is satisfied
  \[ S_r \cap A^{- 1} (V_f) = \emptyset, \]
  or equivalently $f \nin {\rm cl} A (S_r)$. Since there is $\epsilon
  > 0$ such that
  \[ \deg (A, B_r, f) = \deg (A + \epsilon J, B_{r_{}}, f), \]
  it is enough to show 
  \begin{equation}
    \label{G-AdeltaJ} \deg (A + \epsilon J, B_r, f) \neq 0,
  \end{equation}
  for sufficiently large $r$ and sufficiently small $\epsilon > 0$. First, we show
  \[ \deg (A + \epsilon J, B_r, 0) \neq 0. \]
  In fact, if $( A + \epsilon J) ( z) = 0$ for $z \in \partial B_r$, then
  \[ \langle A ( z) - A ( 0), z \rangle + \epsilon \| z \|^2 + \langle A (
     0), z \rangle = 0. \]
  Since $A$ is monotone, the inequality $\epsilon \| z \|^2 + \langle A (
  0), z \rangle \leq 0$ implies $\epsilon \| z \| \leq \| A ( 0) \|$, that
  is impossible for sufficiently large $r$. Since $A + \epsilon J$ is a map
  of class $( S)_+$, define the homotopy $h ( t) = t A + \epsilon J$. It is
  simply seen that $0 \nin h ( t) ( \partial B_r)$ and then
  \[ \deg ( A + \epsilon J, B_r, 0) = \deg ( h ( t), B_r, 0) = \deg ( J,
     B_r, 0) \neq 0. \]
  The proof of (\ref{G-AdeltaJ}) is completely similar to one presented above.
\end{proof}

\end{document}